\documentclass[a4paper,12pt]{article}
\usepackage[left=2cm, right=2cm, top=2.5cm, bottom=2.5cm]{geometry}
\usepackage{amsmath,amssymb,amsfonts,amsthm,mathrsfs}
\usepackage{color}
\usepackage{nicematrix}
\usepackage{tikz-cd}
\usepackage{verbatim}

\def\XXint#1#2#3{{\setbox0=\hbox{$#1{#2#3}{\int}$}
     \vcenter{\hbox{$#2#3$}}\kern-.5\wd0}}

\newtheorem{theorem}{Theorem}[section]
\newtheorem{prop}[theorem]{Proposition}
\newtheorem{cor}[theorem]{Corollary}
\newtheorem{lemma}[theorem]{Lemma}

\theoremstyle{definition}
\newtheorem{definition}[theorem]{Definition}

\newtheorem{remark}[theorem]{Remark}

\usepackage{dsfont}

\newcommand{\strat}{\bullet}

\usepackage{mathtools}

\DeclarePairedDelimiter{\bra}{(}{)}
\DeclarePairedDelimiter{\pra}{[}{]}

\usepackage{ifthen}
\newlength{\leftstackrelawd}
\newlength{\leftstackrelbwd}
\def\leftstackrel#1#2{\settowidth{\leftstackrelawd}%
{${{}^{#1}}$}\settowidth{\leftstackrelbwd}{$#2$}%
\addtolength{\leftstackrelawd}{-\leftstackrelbwd}%
\leavevmode\ifthenelse{\lengthtest{\leftstackrelawd>0pt}}%
{\kern-.5\leftstackrelawd}{}\mathrel{\mathop{#2}\limits^{#1}}}

\newcommand\wh{\widehat}

\newcommand{\R}{\bbR}

\newcommand{\ee}{\mathrm e}
\newcommand{\ba}{\begin{array}} \newcommand{\ea}{\end{array}}
  
\renewcommand{\div}{\mathop{\mathrm{div}}\nolimits}
\newcommand\oti{\mathord{\otimes}}
\newcommand{\e}{\varepsilon}
\newcommand{\ti}{\mathord{\times}}

\newcommand{\Expectation}{\mathbb E}

\DeclareMathOperator{\range}{range}
\DeclareMathOperator{\rank}{rank}

\newcommand\dr{G}   %
\let\wt\widetilde

\providecommand{\keywords}[1]{\par\medskip\noindent\textbf{Keywords and phrases: } #1}
\providecommand{\MSC}[1]{\par\medskip\noindent\textbf{AMS Mathematics Subject Classification: } #1}

 \def\calE{{\mathcal E}} \def\calF{{\mathcal F}}
 \def\calH{{\mathcal H}} 
  \def\calL{{\mathcal L}}

\def\calV{{\mathcal V}}

\def\sfJ{{\mathsf J}} \def\sfK{{\mathsf K}}

\def\bbA{{\mathbb A}} \def\bbB{{\mathbb B}} 
 \def\bbE{{\mathbb E}} \def\bbF{{\mathbb F}}
\def\bbG{{\mathbb G}}  
\def\bbJ{{\mathbb J}} \def\bbK{{\mathbb K}} 
\def\bbM{{\mathbb M}}  
\def\bbP{{\mathbb P}}  \def\bbR{{\mathbb R}}
\def\bbS{{\mathbb S}}

\def\ssharp{{\sharp^\calH}}
\def\J{\sfJ} %
\def\K{\sfK} %

\usepackage[colorlinks=true, linkcolor=black, citecolor=black, urlcolor=black]{hyperref}

\begin{document}
\title{A geometric formulation of GENERIC stochastic differential equations}
\author{Mark A. Peletier \and Marcello Seri}

\maketitle

\begin{abstract}
	We propose a coordinate-invariant geometric formulation of the GENERIC stochastic differential equation, unifying reversible Hamiltonian and irreversible dissipative dynamics within a differential-geometric framework. Our construction builds on the classical GENERIC or metriplectic formalism, extending it to manifolds by introducing a degenerate Poisson structure, a degenerate co-metric, and a volume form satisfying a unimodularity condition. The resulting equation preserves a particular Boltzmann-type measure, ensures almost-sure conservation of energy, and reduces to the deterministic GENERIC/metriplectic formulation in the zero-noise limit. This geometrization separates system-specific quantities from the ambient space, clarifies the roles of the underlying structures, and provides a foundation for analytic and numerical methods, as well as future extensions to quantum and coarse-grained systems.
\end{abstract}

\MSC{37J55, 37K05, 60H10, 60H30, 53D17, 53C17, 80A05, 82C31, 82C35}

\section{Introduction}

The evolution of real-world dynamical systems exhibits an intricate combination of conservative and dissipative dynamics.
The \emph{General Equation for Non-Equilibrium Reverisble Irreversible Coupling} (GENERIC) framework, as suggested by its acronym, is a formalism that combines these conservative (reversible) and dissipative (irreversible) evolutions of the system in a thermodynamically consistent way.

The framework has its roots in 1984, with the independent definitions of dissipative brackets by Grmela~\cite{Grmela84,Grmela85}, Kaufman~\cite{Kaufman84} and Morrison~\cite{Morrison84,Morrison86}. In both~\cite{Grmela84} and~\cite{Morrison86} independently the authors combine these dissipative brackets with Poisson brackets into a structure that is now known as \emph{metriplectic} or \emph{GENERIC}, a geometric framework that has since been used extensively in plasma physics and fluid dynamics~\cite{Guha2007,Birtea2007}.
The acronym GENERIC itself was introduced a decade later by Grmela and \"Ottinger~\cite{GrmelaOttinger97, OttingerGrmela97}.
A comprehensive review of the GENERIC framework can be found in~\cite{Oettinger05}.

While there are slight differences between the metriplectic and GENERIC frameworks (see Remark~\ref{rmk:non-quadratic-dissipation}) the GENERIC framework has been extended to include stochastic effects, leading to the so-called \emph{GENERIC with fluctuations}~\cite{GrmelaOttinger97,Oettinger05}. However, this stochastic formulation is not fully coordinate-invariant: the condition~\eqref{eq:divergenceJ} below that guarantees the stationarity of a Boltzmann-type measure is not invariant under changes of coordinates, as we discuss in Section~\ref{sec:generic-sde}.

\medskip
The goal of this paper is to propose a geometric formulation of the GENERIC stochastic differential equation (SDE) that is coordinate-invariant, and that reduces to a geometric formulation of the deterministic metriplectic equation in the limit of vanishing noise. This formulation is given in  Definitions~\ref{def:mfGSDE-generator} and~\ref{def:mfGSDE-SDE} below.

\medskip
There are many reasons for pursuing such a goal.
For instance, a geometric formulation can allow to separate the description of the ambient space, in which the dynamics takes place, from the specifics of the system itself, clarifying the fundamental roles of each of the quantities at play in a way that is both mathematically clean and physically meaningful.
Furthermore, such a formulation brings with itself a wealth of new tools from differential geometry and mechanics that can be used to provide new analytic and numerical tools to study qualitative and quantitative properties of the dynamics.
And finally, a geometrization can reduce the gap to the study of the quantum counterparts of these systems, a topic that we plan to explore in the close future.

It should not be a surprise, then, that much effort has gone into geometric approaches to describe dissipative systems.
For instance, contact mechanics was connected to dissipative mechanics and GENERIC itself very early on~\cite{Cantrijn1982,Mrugala1991,Mrugala1993} but only recently has seen a surge of interest both as a standalone theory \cite{Grmela2014,Bravetti2017,deLeon2017,Ciaglia2018,Gaset2020, Colombo2025} and as a possible framework to geometrize the GENERIC equation~\cite{EsenGrmelaPavelka22I,EsenGrmelaPavelka22II}.
However, these are by no means the only possible geometries, as exemplified by the recent works geometrizing dissipation with b-symplectic geometry \cite{Coquinot2023} or Jacobi-Haantjes manifolds \cite{azuaje2025}.
An important step in this geometrization is constructing or leveraging the appropriate mathematical tools to deal with the inherent singularities that arise in the description of dissipative systems, as we will see in Section~\ref{sec:geometrization}.

\medskip
The paper is structured as follows. In Section~\ref{sec:generic} we recall the definition of the GENERIC equation in the Euclidean setting and its main properties, and define our main notation and terminology. In Section~\ref{sec:generic-sde} we introduce the stochastic version of the GENERIC equation in the Euclidean setting, and we discuss in particular the conditions that guarantee the existence of a stationary measure. In Section~\ref{s:sdes-on-mfs} we describe the basic theory of stochastic differential equations on manifolds, and in Section~\ref{sec:geometrization} we present our geometric formulation of the GENERIC SDE, discussing its properties and its relation to the Euclidean formulation and the deterministic metriplectic equation.
In Section~\ref{s:vanishing-noise} we discuss the `zero-noise limit' and show that the geometric GENERIC SDE reduces to the deterministic GENERIC equation in that limit.
We close the paper in Section~\ref{sec:discussion} with some concluding remarks.

\section{GENERIC}\label{sec:generic}

We will start our discussion in the Euclidean setting in order to connect to the usual formulation of GENERIC in the original work~\cite{GrmelaOttinger97,OttingerGrmela97,Oettinger05}. We then generalize the formalism leading to the geometrization of Section~\ref{sec:geometrization}.

Within the GENERIC framework, the evolution of the state $x(t) \in \bbR^d$ of a system takes the form
\begin{equation}\label{eq:generic}
	\dot x = J(x)\, DE(x) + K(x)\, DS(x),
\end{equation}
where $D$ is the total derivative of the real functions $E, S : \bbR^d \to \bbR$ and $J, K: \bbR^d \to \mathrm{Mat}(d,\bbR)$ are matrix-valued functions with the following properties:
\begin{enumerate}
	\item\label{cond:generic-metric} For all $x\in\bbR^d$, $K(x)$ is a $d \times d$ symmetric  positive semidefinite matrix, and plays the role of a (degenerate) inner product.
	\item\label{cond:generic-poission} For all $x\in\bbR^d$, $J(x)$ is a $d\times d$ antisymmetric matrix satisfying the Jacobi identity in the following sense. Let $F,G,H \in C^\infty(\bbR^d)$, if $\{F,G\}_J := \langle DF, J\, DG \rangle$, then
	      \[
		      \{F,\{G,H\}_J\}_J + \{G,\{H,F\}_J\}_J + \{H,\{F,G\}_J\}_J = 0.
	      \]
	      This plays the role of a (degenerate) Poisson structure.
	\item\label{cond:generic-degeneracy} The following \emph{non-interaction conditions} are satisfied: for all $x\in\bbR^d$,
	      \begin{equation}
		      J(x)\, DS(x)  = 0 \qquad\text{and}\qquad
		      K(x)\, DE(x)  = 0.
		      \label{eq:NIC-Rd}
	      \end{equation}
\end{enumerate}
The physical origins of the GENERIC approach are reflected in the names commonly associated to some of the functions just introduced, where $E$ is called the \emph{total energy} of the system and $S$ its \emph{entropy}.

If we consider $K$ to be the matrix form of a degenerate inner product, $K(x)\,D$  represents the associated gradient and $\dot x = K(x)\, DS(x)$ the corresponding gradient flow.
In a similar vein, $J$ is a degenerate Poisson matrix, giving rise to the Hamiltonian vector field $J(x)\, DE(x)$. That is, $\dot x = J(x)\, DE(x)$ is Hamilton's equation for the Hamiltonian function $E$ with respect to the Poisson structure induced by $J$.
This already shows that the theory generalizes both classical Hamiltonian systems and gradient flows: these correspond respectively to the special cases $K(x)\, DS(x) = 0$ and $J(x)\, DE(x) = 0$.

\begin{remark}[Notational convention for derivatives in $\bbR^d$]
	When working in $\bbR^d$, we do not distinguish between vectors and co-vectors, silently identified via the standard inner product, and we denote the total derivative and the gradient by the same symbol~$D$.
	Given $x = (x_1, \ldots, x_d) \in \bbR^d$, we denote by $\partial_{x_i}$ the partial derivative with respect to the $i$-th coordinate.
\end{remark}

\begin{remark}[Historical remarks and terminology]
	In geometric mechanics and Poisson geometry, phase-space functions that are conserved by the dynamics associated to any Hamiltonian (itself a phase-space function) are called \emph{Casimirs} of the Poisson structure~\cite{Marsden_1999}.
	In these terms, the condition on $J$ at point \ref{cond:generic-degeneracy} above says that entropy $S$ is a Casimir of the Poisson structure $J$.
	In analogy to this, Kaufman~\cite{Kaufman84} introduced a dissipative bracket associated to the degenerate inner product, defining \emph{dissipative invariants} as the phase-space functions whose dissipative brackets with any entropy (itself again a phase-space function) are vanishing.
	In this sense, the condition on $K$ at point \ref{cond:generic-degeneracy} above is rephrased saying that energy $E$ is a dissipative invariant of the dissipative structure $K$, also called a ``dissipation Casimir''~\cite{EsenGrmelaPavelka22I}.
\end{remark}

At a basic level, the thermodynamic consistency of GENERIC can be recognized in the behaviour of the energy and entropy functions $E$ and $S$.
Indeed, let $x(t)$ be a solution of \eqref{eq:generic}. Then, the degeneracy conditions and the properties of $J$ and $K$ imply
\begin{align*}
	\frac{d}{dt} E(x(t)) & = \langle DE(x(t)), \dot{x}(t)\rangle = \langle DE, J\, DE \rangle + \langle DE, K\, DS \rangle = 0,     \\
	\frac{d}{dt} S(x(t)) & = \langle DS(x(t)), \dot{x}(t) \rangle = \langle DS, J\, DE \rangle + \langle DS, K\, DS \rangle \geq 0.
\end{align*}
That is, energy is conserved along the Hamiltonian evolution, while entropy is non-decreasing along the gradient flow evolution, in accordance with the first and second principles of thermodynamics.

\begin{remark}[Finite dimensions]
	The only reason we restrict ourselves to the finite-dimensional Euclidean setting is for simplicity of exposition. The extension to systems defined on Hilbert spaces is relatively straightforward, and the properties discussed above remain valid also in such a setting \cite{MielkePeletierZimmer25}.
\end{remark}

\begin{remark}[Non-quadratic dissipation]\label{rmk:non-quadratic-dissipation}
	We only presented the GENERIC formulation in a large but specific sub-case.
	An important generalization consists in replacing the linear operator $\xi \mapsto K(x)\xi$ by a nonlinear operator $\xi \mapsto \overline K(x,\xi)$. The natural generalization of the symmetry and non-negativity conditions on $K$ are that for each $x$, $\overline K(x,\cdot)$ is the derivative of a convex lower semicontinuous function. Such generalizations have been proposed on the basis of the structure of the Boltzmann collision term~\cite{Grmela93} and chemical reaction kinetics~\cite{Grmela10}, and have been shown to arise as large-deviation rate functions~\cite{MielkePeletierRenger14,KraaijLazarescuMaesPeletier20}. In this paper we restrict ourselves to the case of linear $K$, which lends itself to a more straightforward geometric interpretation, and we postpone the discussion of the more general case to future work.
\end{remark}

\section{GENERIC SDEs}\label{sec:generic-sde}

To clarify some of the choices that enter in the geometric formulation of GENERIC that we present below, we first discuss  `GENERIC with fluctuations', which is the stochastic version of GENERIC presented in~\cite{GrmelaOttinger97,Oettinger05}.
This is a stochastic differential equation (SDE) for a random process $X_t$ taking values in $\bbR^d$ of the form
\begin{equation}\label{eq:generic-sde}
	dX_t = \bra[\Big]{J(X_t)\, DE(X_t) + K(X_t)\, DS(X_t) +  \div K(X_t)}\, dt + \Sigma(X_t) \, dB_t,
\end{equation}
where $B_t$ is a standard Brownian motion and where we are assuming that~$K$ and the mobility~$\Sigma$ are related by the following \emph{fluctuation-dissipation relation}:
\begin{equation}\label{eq:fluctuation-dissipation}
	\Sigma \Sigma^*(x) = 2  K(x) \quad \text{for all } x \in \bbR^d.
\end{equation}
Historically this SDE has been introduced by~\cite{GrmelaOttinger97} and is found in the literature also under the name of ``GENERIC with fluctuations''~\cite{Oettinger05}. A well-known argument based on It\^o's Lemma~\cite[Sec.~3.4]{Pavliotis14} shows that the process has as generator
\begin{equation}
	\label{eqdef:generator-sde-Euclidean}
	Lf = Df^*JDE + Df^*KDS +\div (KD f).
\end{equation}

Because of the fluctuation-dissipation relation~\eqref{eq:fluctuation-dissipation} and the non-interaction conditions~\eqref{eq:NIC-Rd}, the SDE~\eqref{eq:generic-sde} inherits the conservation of energy~\cite[Lemma 4.5, part 1]{MielkePeletierZimmer25}:
\begin{lemma}
	Assume that $E,S,J,$ and $K$ satisfy conditions~\ref{cond:generic-metric}--\ref{cond:generic-degeneracy} of the deterministic GENERIC equation in Section~\ref{sec:generic} and that \eqref{eq:fluctuation-dissipation} holds. Then the evolution deterministically preserves the energy $E$, that is, for any $X_0\in\bbR^d$,
	\[
		E(X_t) = E(X_0) \quad\mbox{almost surely, for all } t \geq 0.
	\]
\end{lemma}

A central step in our approach is to require that the measure $\mu(dx)=e^{S(x)}\, dx$ is stationary under the SDE~\eqref{eq:generic-sde}. This property is implicitly present in various places in the literature, such as~\cite[(54)]{GrmelaOttinger97}, \cite[(10)]{Ottinger98}, and~\cite[(5.44c)]{PavelkaKlikaGrmela18}, and we discuss it in more detail in Section~\ref{ss:props-of-sde-mf} below.

To achieve this stationarity, an additional constraint on the GENERIC components is required.
In~\cite{MielkePeletierZimmer25}, the authors propose the additional condition
\begin{equation}\label{eq:divergenceJ}
	\div J = 0,
\end{equation}
where we take the divergence of a matrix field to be the column vector of the divergences of each of the rows.
This assumption leads to the following result~\cite[Lemma 4.5, part 2]{MielkePeletierZimmer25}.
\begin{lemma}\label{lem:stationary-measure-GENERIC}
	Under the conditions of the previous lemma and~\eqref{eq:divergenceJ}, the measure $\mu = e^{S(x)}\, dx$ is stationary for the SDE~\eqref{eq:generic-sde}.
\end{lemma}
\noindent
We prove a more general version of this lemma below as Lemma~\ref{l:invariance-ebetaS}.

\begin{remark}
	An alternative proposal was put forward by \"Ottinger~\cite[(6.163)]{Oettinger05},  in which~\eqref{eq:divergenceJ} is replaced by a relation between $J$ and $S$. In this paper we do not further consider this proposal, both because $S$ is typically assumed to be independent of temperature and because  it is not clear how to extend this to a geometric formulation of GENERIC in a coordinate-free manner.
\end{remark}

A problem with the condition~\eqref{eq:divergenceJ} is that in current form it is not invariant under changes of coordinates. In~\cite[Remark 4.7]{MielkePeletierZimmer25} the authors recognize this and propose a generalization involving the existence of a density $m$ on $\R^d$ such that $\div(m J) = 0$.
This interpretation is already more natural as it leads to a coordinate-invariant description that can be directly connected to uni-modularity of the Poisson structure $J$.
In some sense, this already hints at the possibility of a geometrization of the condition~\eqref{eq:divergenceJ}. However, also in this case, the connection to a geometric formulation of GENERIC is not immediate, since the coordinate invariance comes from a non-trivial interplay between the condition above, an It\^o correction term, and the gradient term $K\,DS$.

\section{Stochastic differential equations on manifolds}
\label{s:sdes-on-mfs}

As preparation for the definition of the GENERIC SDE in the manifold context we first describe the basic theory of SDEs on manifolds. In this paper we follow Thalmaier's presentation~\cite{Thalmaier16}, and we assume that the reader is familiar with the theory of stochastic processes in $\R^d$, in particular with the concepts of a filtered probability space, semimartingales, stopping times, adapted processes,  conditional expectations, and It\^o and Stratonovich stochastic differential equations in $\R^d$.  Good introductions to these can be found in many places, e.g.~\cite{IkedaWatanabe89,RevuzYor13, Protter13, Evans_2013}.

Similarly, we also assume that the reader is familiar with the  basic concepts of differential geometry of smooth manifolds, such as tensor calculus, metric and Poisson structures, Hamiltonian and gradient flows. Good introductions to these can also be found in many places, e.g.~\cite{Abraham_1988,Marsden_1999,Lee_2012}.

\subsection{Definition of an SDE through Stratonovich integration}
Fix a {filtered probability space $\bbF:=(\Omega,\calF,(\calF_t)_{t\geq0},\bbP)$ satisfying the usual conditions}.
\begin{definition}[Manifold SDE; {\cite[Sec.~2.3]{Thalmaier16}}]
	\label{def:SDE-mf-Thalmaier}
	Fix a differential manifold $M$ and a real finite-dimensional space $E$.
	A stochastic differential equation on $M$ is characterized by a pair $(\dr,Z)$, where
	\begin{enumerate}
		\item $Z$ is a {semimartingale} on $\bbF$ with values in $E$, and
		\item $\dr: M\ti E \to TM$ is a smooth homomorphism of vector bundles, i.e.\ a smooth map such that the following diagram commutes:
		      \begin{center}
			      \begin{tikzcd}
				      M \ti E \arrow[r, "\dr"] \arrow[d, "\mathrm{pr}_1"] & TM \arrow[d, "\pi"] \\
				      M \arrow[r, "\mathrm{id}"] & M
			      \end{tikzcd}
			      \quad.
		      \end{center}
	\end{enumerate}
	We write $\dr(x): E \to T_x M$ for the map $e\mapsto \dr(x) e := \dr(x,e)$.
	With this definition, $\dr(\cdot)e \in \Gamma(TM)$ for all $e\in E$.
\end{definition}

The pair $(\dr,Z)$ generates the {Stratonovich} stochastic differential equation on $M$,
\begin{equation}
	\label{eqdef-general:SDE-on-M}
	dX_t = \dr(X_t) \strat dZ_t,
\end{equation}
with the following solution concept.

\begin{definition}[Solution of the SDE; {\cite[Def.~2.15]{Thalmaier16}}]
	\label{def:solution-of-SDE}
	Let $(\dr,Z)$ be as in Definition~\ref{def:SDE-mf-Thalmaier}.
	Let $(X_t)_{0\leq t< \zeta}$ be an adapted process  with continuous sample paths, defined up to a {stopping time} $\zeta\geq0$. We say that $X$ is a solution of the SDE~\eqref{eqdef-general:SDE-on-M} if the following two properties hold for all $f\in C^\infty(M)$:
	\begin{itemize}
		\item $f\circ X$ is a semimartingale;
		\item For any {stopping time} $0\leq \tau < \zeta$,
		      \begin{equation}
			      \label{eqdef:soln-of-SDE-Ito-formula}
			      f(X_\tau) = f(X_0) + \int_0^\tau \pra[\big]{(df)_{X_s}\dr(X_s)}\strat dZ_s.
		      \end{equation}
	\end{itemize}
\end{definition}
\noindent
Note that the classical It\^o-Stratonovich formula has been promoted from consequence to definition, in the form of~\eqref{eqdef:soln-of-SDE-Ito-formula}.
Recalling that $X_s$ takes values on the manifold $M$, the argument $(df)_{X_s}\dr(X_s)$ of the integral is a linear map $E\to T_{X_s}M \to \R$, and the integral is then a classical Stratonovich integral in $\R$.

\begin{theorem}[Existence and uniqueness of solutions of the SDE; {\cite[Th.~2.20]{Thalmaier16}}]
	\label{t:ex-in-SDE}
	Let $M$, $\dr $, and $Z$ be as above, and fix an $\mathscr{F}_0$-measurable $x_0: \Omega \rightarrow M$.
	Then there exists a unique maximal solution $(X_t)_{0\leq t< \zeta}$  of the SDE
	$$
		d X=\dr (X) \strat d Z
	$$
	with initial condition $X_0=x_0$. Here the stopping time $\zeta$ is almost surely strictly positive. Uniqueness holds in the sense that if $(Y_t)_{0\leq t<\xi}$ is another solution with $Y_0=x_0$, then almost surely $\xi \leq \zeta$ and $(X_t)_{0\leq t<\xi} = (Y_t)_{0\leq t<\xi}$.
\end{theorem}

\medskip
The right-hand side in~\eqref{eqdef-general:SDE-on-M} allows for both deterministic `drift' terms and stochastic `noise' terms. We will use the following special case in the sequel: take  $E = \R^{r+1}$ and assume that the homomorphism $\dr: M\ti E\to TM$ is given. We take $Z = (t,W^1,\dots , W^r)$ where $(W^1,\dots,W^r)$ is a vector of independent scalar standard Brownian motions. Defining $\dr_i$ in terms of the elementary unit vectors $e_i$  by
\[
	\dr_i: M\to TM, \qquad
	\dr_i(x) := \dr (x)e_i, \qquad i=0,1,\dots,r,
\]
we can write the equation~\eqref{eqdef-general:SDE-on-M} as
\begin{equation}
	\label{eqdef:SDE-on-mf-driven-by-BM}
	dX_t = \dr_0(X_t) \, dt + \sum_{i=1}^r \dr_i(X_t)\strat dW^i_t.
\end{equation}
Here one recognizes the first term as a `drift' term. For this choice of $Z$, by localizing to a chart one infers that solutions of~\eqref{eqdef:SDE-on-mf-driven-by-BM} are almost surely continuous, similarly to solutions of SDEs in $\R^d$.

\subsection{Generator characterization}

Solutions of the SDE can also be characterized in terms of their infinitesimal generator:
\begin{lemma}[Ex.~2.17 and Cor.~2.18 of~\cite{Thalmaier16}]
	\label{l:SDE-is-L-diffusion}
	If the process defined by the SDE~\eqref{eqdef:SDE-on-mf-driven-by-BM} has infinite lifetime, i.e.\ $\zeta = +\infty$ almost surely, then the process is an \emph{$L$-diffusion}, where the operator $L$ is given by
	\begin{equation}\label{eq:generator-L}
		Lf = \dr _0f + \frac12 \sum_{i=1}^r \dr _i(\dr _if)
		\qquad \text{for any }f\in C_c^\infty(M).
	\end{equation}
\end{lemma}

\begin{definition}[$L$-diffusions; {\cite[Def.~1.3]{Thalmaier16}}]
	\label{def:L-diffusions}
	An adapted continuous process $(X_t)_{t\geq0}$  taking values in $M$ is called an \emph{$L$-diffusion} with starting point $x$ if $X_0(x)=x$ and if, for all $f \in C_c^{\infty}(M)$, the process
	\[
		N_t^f(x):=f\left(X_t(x)\right)-f(x)-\int_0^t(L f)\left(X_s(x)\right) d s, \quad t \geq 0
	\]
	is a martingale, i.e.
	\begin{equation}
		\label{eq:def:L-diffusion:mg-char}
		\mathbb{E}^{\mathscr{F}_s} \left[f\left(X_t(x)\right)-f\left(X_s(x)\right)-\int_s^t(L f)\left(X_r(x)\right) d r\right]=0, \quad \text { for all } 0\leq s \leq t.
	\end{equation}
\end{definition}

\noindent
From~\eqref{eq:def:L-diffusion:mg-char} it follows directly that $L$ is the infinitesimal generator of $X$ in the usual sense:
\begin{cor}[Alternative characterization of $L$]
	It follows that for all $f\in C_c^\infty(M)$
	\[
		Lf(x) = \lim_{t\downarrow 0}\frac1t \Expectation\pra[\Big]{\,f(X_t)-f(X_0)\,\Big|\, X_0=x\,}.
	\]
\end{cor}

\begin{remark}[Constructing the SDE from the generator]
	Given a generator of the form~\eqref{eq:generator-L}, one can straightforwardly read off the components $G_0,\dots,G_r$ and build the SDE~\eqref{eqdef:SDE-on-mf-driven-by-BM}. Therefore the generator also completely specifies the SDE~\eqref{eqdef:SDE-on-mf-driven-by-BM}.
\end{remark}

\begin{remark}[Finite lifetime]
	The restriction in Lemma~\ref{l:SDE-is-L-diffusion} to processes with infinite lifetime is necessary, since  if the process $X$ can escape to infinity in finite time, then the probability that $X_t$ is in $M$ eventually becomes strictly smaller than $1$. This can be remedied by extending~$M$ with a `graveyard' state, but  since this provides no additional insight into the geometry of the SDE we do not pursue this here.
\end{remark}

\bigskip

In conclusion, there are two ways to characterize stochastic differential equations on a manifold $M$:
\begin{itemize}
	\item Directly as solutions of a Stratonovich SDE, as in Definition~\ref{def:solution-of-SDE};
	\item Indirectly via their infinitesimal generator $L$, as in Definition~\ref{def:L-diffusions}.
\end{itemize}
Lemma~\ref{l:SDE-is-L-diffusion} states that if the lifetime is infinite, then these two are equivalent (and have unique solutions).

\section{Geometrization}\label{sec:geometrization}

We now formulate a geometrically consistent version of the GENERIC SDE~\eqref{eq:generic-sde}. As described in Section~\ref{sec:generic-sde}, the condition~\eqref{eq:divergenceJ} is not coordinate-invariant, and in the setup of this paper we remedy this by introducing a volume form $\nu$. We comment on this choice in more detail in Section~\ref{ss:unimodularity}.

\subsection{Components of the geometric GENERIC SDE}
In this paper, a geometric GENERIC SDE (gGENERIC SDE) is characterized by:
\begin{itemize}
	\item A smooth finite-dimensional manifold $M$ without boundary,
	\item a \emph{degenerate co-metric} tensor $\K : T^*M \oti T^*M \to \mathbb{R}$, that is, a non-negative definite, symmetric contravariant 2-tensor field on $M$,
	\item a \emph{degenerate Poisson structure} $\J : T^*M \oti T^*M \to \mathbb{R}$, that is, a closed, non-negative definite, contravariant alternating 2-tensor field on $M$ whose associated Poisson bracket of smooth functions $\{f,g\} := \J(df,dg)$ satisfies the Jacobi identity,
	\item  \emph{a non-degenerate volume form} $\nu$ on $M$, and
	\item two \emph{functions} $E, S \in C^\infty(M)$.
\end{itemize}
These are required to satisfy the additional conditions
\begin{itemize}
	\item \emph{Non-interaction} of $\J$ with $S$ and of $\K$ with $E$:
	      \begin{equation}
		      \label{eqdef:NIC-mf}
		      \J(dS, \cdot) = 0,
		      \quad
		      \K(dE, \cdot) = 0.
	      \end{equation}
	\item \emph{Unimodularity} of $\nu$:
	      \begin{equation}
		      \label{eq:ass:unimodularity-mf}
		      \div_\nu (\J(dh,\cdot)) = 0 \qquad \text{for all } h \in C^\infty(M).
	      \end{equation}
\end{itemize}
Here the divergence $\div_\nu$ is defined by $\calL_X \nu = \div_\nu(X)\, \nu$ for all $X\in\Gamma(TM)$, and it satisfies the following integration-by-parts formula.

\begin{lemma}[Integration by parts]
	In the setup above, for any $f\in C^\infty_c(M)$ and any vector field $X$ on $M$, we have
	\begin{equation}
		\label{eq:int-by-parts}
		\int _M X(f)\, \nu = -\int_M f\div_\nu (X) \, \nu.
	\end{equation}
\end{lemma}

\begin{proof}
	Recall that the Lie derivative $\calL_{X}$ satisfies the general identity
	\[
		\calL_{X}(f \nu) = (\calL_{X} f) \nu + f (\calL_{X} \nu) = X(f) \nu + f (\calL_{X} \nu).
	\]
	Using Cartan's magic formula, we find
	\[
		\calL_{X}(f \nu) = d(\iota_{X} (f\nu)) + \iota_{X}(d(f\nu)) = d(\iota_{X} (f\nu)),
	\]
	the last identity following from the fact that $\tilde{\nu} := f\nu$ is a volume form (has top degree) and thus $d \tilde{\nu} = 0$.
	By Stokes' theorem we therefore find $\int_M \calL_X(f\nu)=0$, since $M$ has no boundary, and therefore
	\[
		0 = \int_M \calL_X(f\nu) = \int_M X(f)\, \nu + \int_M f(\calL_X\nu)
		= \int_M X(f)\, \nu + \int_M f\div_\nu (X) \, \nu,
	\]
	which proves~\eqref{eq:int-by-parts}.
\end{proof}

Before entering into the details of defining the stochastic gGENERIC evolution, observe that we can already formulate the deterministic gGENERIC equation on $M$ as
\begin{equation}
	\label{eq:gGENERIC-v2}
	\dot x = \J(dE, \cdot)|_x + \K(dS, \cdot)|_x,
\end{equation}
or in equivalent but more explicit form,
\begin{equation}
	\label{eq:gGENERIC}
	\frac d{dt} h(x(t)) = \J(dE, dh)(x(t)) + \K(dS, dh)(x(t)),
	\quad \text{for all } h\in C^\infty(M).
\end{equation}
These two equations should be seen as simple reformulations of~\eqref{eq:generic} in slightly different notation.

\bigskip
The conditions on $\K$, $\J$, $E$, and $S$, are all natural generalizations of the corresponding ones in Section~\ref{sec:generic}. The volume form $\nu$ and the unimodularity assumption, however,  are new additions, and they can be interpreted as follows.
For a given vector field~$X$, the statement $\div_\nu X = 0$ is equivalent to the preservation of the volume $\nu$ by the flow of~$X$~\cite[Proposition 6.5.18]{Abraham_1988}.
The unimodularity assumption therefore can be restated as the property that $\nu$ is conserved by all Hamiltonian vector fields $X_h :=\J(dh,\cdot)$ generated by functions  $h$.
In particular, it is important to emphasize at this point that the unimodularity condition as stated here, is a coordinate-free generalization of the condition $\div J = 0$ in~\eqref{eq:divergenceJ}, as we discuss in more detail in Remark~\ref{rmk:univalence-geometric} below.

\medskip

An important consequence of the unimodularity assumption~\eqref{eq:ass:unimodularity-mf} is the following proposition. This result is contained in \cite[Chapter 9.3]{Crainic2021} in the broader context of Poisson cohomology, but here we give a statement and proof that only require basic differential geometry.
\begin{prop}\label{prop:unimodularity}
	Let $M$ be a smooth manifold without boundary, $f \in C^\infty(M)$, and $h\in C_c^\infty(M)$.
	Then
	\[
		\int_M \J(df, dh)\nu = 0.
	\]
\end{prop}
\begin{proof}
	Let $X_f$ denote the Hamiltonian vector field associated to $f$ via the Poisson structure~$\J$, that is, for all smooth functions $h$, $X_f (h):= \J(df, dh)$.
	Then
	\begin{align}
		\int_M \J(df, dh) \nu & = \int_M X_f(h)\nu \nonumber
		\stackrel{\eqref{eq:int-by-parts}}= - \int_M h \div_\nu(X_f) \nu. \label{eq:int-J}
	\end{align}
	The unimodularity assumption \eqref{eq:ass:unimodularity-mf} guarantees that the right-hand side vanishes for all $f$, concluding the proof.
\end{proof}

Since we will be introducing an SDE, we fix for once and for all a filtered probability space $(\Omega,\calF,(\calF_t)_{t\geq0},\bbP)$ as in Section~\ref{s:sdes-on-mfs}.

\subsection{Sub-Riemannian Brownian motion generated by $\K$}
\label{ss:sub-Riemannian-BM}

We first construct the manifold equivalent of the Euclidean operator $f\mapsto \div (K\nabla f)$. The canonical generalisation to a Riemannian manifold would be the Laplace-Beltrami operator $\Delta_M := \div_g \circ \nabla_g$.
To adapt the theory to our setting of a degenerate co-metric $\K$, we need a canonical way to construct such operators.
Thankfully, the necessary theory has already been developed in the context of hypoelliptic diffusion and, more recently, sub-Riemannian geometry~\cite{Agrachev_2019}.
A self-contained introductory reference is \cite{Thalmaier16}.

To define the gradient and divergence operators, we need to introduce a metric and a volume.
We first construct the metric.
Define a map $\ssharp: T^*M \to \calH\subseteq TM$ by the property that for all $\alpha \in \Gamma(T^*M)$,
\[
	\ssharp \alpha = \K(\alpha, \cdot),
\]
where $\calH := \range \ssharp \subset TM$ is also a subbundle of $TM$.
With this at hand, we can define the \emph{(sub-Riemannian) metric} tensor $g_\calH$ on $\calH\subseteq TM$ by
\[
	g_\calH(\ssharp \alpha, \ssharp \beta) = \K(\alpha, \beta),
	\qquad\text{for all }\alpha, \beta \in \Gamma(T^*M).
\]
The subbundle $\calH$ is called the \emph{horizontal distribution} and the triple $(M, \calH, g_\calH)$ is commonly known as a \emph{sub-Riemannian manifold}.

\begin{remark}
	Note that it is common in sub-Riemannian geometry to assume that the distribution $\calH$ is \emph{bracket-generating}, that is, the Lie algebra generated by the horizontal vector fields and their iterated Lie-brackets spans $TM$. We are \textbf{not} making this assumption here, as it is not necessary for the definition of the sub-Laplacian or the stochastic processes we are interested in, and in our application to GENERIC SDEs it will in fact never be satisfied (see Remark~\ref{rem:energy-is-conserved}).
\end{remark}

With this metric, the \emph{(horizontal) gradient} of a function $f \in C^\infty(M)$ is defined as the unique vector field $\nabla_\calH f \in \Gamma(TM)$ such that
\begin{equation}
	\label{eqdef:gradient-H}
	g_\calH(\nabla_\calH f, X) = X(df), \quad \text{for all } X \in \Gamma(\calH).
\end{equation}
In particular,
\[
	\nabla_\calH f = \ssharp df = \K(df,\cdot)
	\qquad\text{and}\qquad
	g_{\calH}(\nabla_\calH f, \nabla_\calH g) = \K(df,dg),
\]
and the integration-by-parts formula~\eqref{eq:int-by-parts} implies that if either $f$ or $h$ has compact support, then
\begin{equation}
	\label{eq:int-by-parts-K}
	\int_M f \div_\nu (\nabla_\calH h )\, \nu
	= -\int_M  \K(df,dh)\, \nu.
\end{equation}

For the volume we choose the volume form $\nu$ that was introduced at the start of this section.
Indeed, the degeneracy condition~\eqref{eqdef:NIC-mf} implies that the tensor $\K$ is not invertible and therefore there is no equivalent of a Riemannian volume. This is a mathematical indication that the choice of the volume form $\nu$ is indeed a choice in which we have a certain freedom. In fact, we claim that this choice is a \emph{modelling} choice, and we comment on this in more detail in Sections~\ref{ss:unimodularity} and~\ref{ss:math-vs-modelling}.

Given the metric $g_\calH$ and the volume form $\nu$, the sub-Laplacian $\Delta_\calH f$ is the operator on $L^2(M; \nu)$ with domain $C_c^\infty(M)$ defined as the divergence of the horizontal gradient, that is,
\begin{equation}\label{eq:sub-Laplacian}
	\Delta_\calH f = \div_\nu(\nabla_\calH f) = \div_\nu(\ssharp df).
\end{equation}

It is possible to give a more explicit local characterization of $\Delta_\calH$ in terms of a local orthonormal frame (pointwise local orthonormal basis) of $\calH$, see e.g.~\cite[Remark 1.30]{Laurent2022} or~\cite[Section 2.2]{Agrachev_2009}, as follows.
\begin{lemma}\label{lem:DeltaH-local}
	Given a local horizontal frame of $\calH$, $\{A_1, \ldots, A_r\}\subset \Gamma(\calH)$, the operator $\Delta_\calH$ can always be expressed as the H\"ormander type operator
	\begin{equation}
		\label{eq:char-DeltaH-Ai}
		\Delta_\calH f = -\sum_{i=1}^r A_i^* A_i f = \sum_{i=1}^r A_i (A_i f) + A_0 f, \qquad r =\rank \calH,
	\end{equation}
	where $A_i^* = -A_i - \div_\nu(A_i)$ is the adjoint of $A_i$ in $L^2(M;\nu)$ and $A_0 := \sum_{i=1}^r \div_\nu(A_i) A_i$.
\end{lemma}
A consequence of the setup above is the identity
\begin{equation}
	\label{eq:relation-Ai-K}
	\sum_{i=1}^d (A_if) (A_ih)
	= g_\calH(\nabla_\calH f, \nabla_\calH h)
	= \K(df,dh) \qquad \text{for all }f,h\in C^\infty(M).
\end{equation}

The operator $\Delta_\calH$ is a symmetric (unbounded) operator on $L^2(M;\nu)$, since for all $f,h\in C_c^\infty(M)$ we have
\begin{align}
	\label{eq:char:DeltaH-Ais}
	(f,\Delta_\calH h)_{L^2_\nu}
	 & = \int_M f (\Delta_\calH h) \,\nu \stackrel{\eqref{eq:char-DeltaH-Ai}}= -\sum_{i=1}^r \int_M h \, A_i^* A_i f \,\nu
	= -\sum_{i=1}^r \int_M (A_i h) (A_i f) \,\nu.
\end{align}
It can be interpreted as the derivative of the quadratic (Dirichlet) form
\begin{equation}
	\label{eqdef:Dirichlet-form-1}
	\calE(f,f) = \frac12 \int_M g_\calH(\nabla_\calH f, \nabla_\calH f) \,\nu.
\end{equation}
As we saw in Theorem~\ref{t:ex-in-SDE} and Lemma~\ref{l:SDE-is-L-diffusion}, there exists a diffusion process in $M$ whose generator is the operator~$\Delta_\calH$.

\subsection{Defining the geometric GENERIC SDE}

Recall from the start of this section that we are given a manifold $M$, tensors $\K$ and $\J$, functionals $E$ and~$S$, and a volume form $\nu$, with various properties. The previous section introduces the corresponding operator $\Delta_\calH$.

We now define the geometric GENERIC SDE (gGENERIC SDE) in the two equivalent ways that are described at the end of Section~\ref{s:sdes-on-mfs}.

\begin{definition}[gGENERIC SDE, generator definition]
	\label{def:mfGSDE-generator}
	The gGENERIC SDE is the stochastic process on $M$ with infinitesimal generator
	\begin{equation}
		\label{eq:generator-mf}
		Lf = \J(dE,df) + \K(dS,df) + \Delta_\calH f \qquad\text{for any }f\in C_c^\infty(M),
	\end{equation}
	if it exists.
\end{definition}
\noindent
Note how this expression mirrors the generator~\eqref{eqdef:generator-sde-Euclidean} in the Euclidean case.

\begin{definition}[gGENERIC SDE, SDE definition]
	\label{def:mfGSDE-SDE}
	Recall from Section~\ref{ss:sub-Riemannian-BM} the vector fields $A_i$ that characterize the operator $\Delta_\calH$. Define the vector field $B_0$ by
	\begin{align}
		B_0 & = \J(dE,\cdot) + \K(dS,\cdot) +  \sum_{i=1}^d (\div_\nu  A_i) A_i
		= \J(dE,\cdot) + \K(dS,\cdot) + A_0,
	\end{align}
	The gGENERIC SDE then is the Stratonovich SDE
	\begin{subequations}
		\label{eqdef:two-versions-of-gGENERIC-SDE}
		\begin{equation}
			\label{eqdef:GENERIC-sde-mf}
			d X_t=B_0 (X_t)\, dt + \sqrt 2 \sum_{i=1}^d A_i(X_t) \strat dW^i_t,
		\end{equation}
		or equivalently in one line
		\begin{equation}
			\label{eqdef:sde-mf-with-JK}
			d X_t= \bra[\Big]{\J(dE,\cdot) + \K(dS,\cdot) +  \sum_{i=1}^d (\div_\nu  A_i) A_i}\, dt + \sqrt{2}\sum_{i=1}^d A_i(X_t) \strat dW^i_t,
		\end{equation}
	\end{subequations}
	where $W^i$, $i=1, \dots, d$, are independent standard Brownian motions.
\end{definition}

The following lemma connects the two formulations.
\begin{lemma}
	If the process~\eqref{eqdef:GENERIC-sde-mf} has infinite lifetime ($\zeta=+\infty$ almost surely) then its infinitesimal generator is the operator $L$ defined in~\eqref{eq:generator-mf}, and the process is an $L$-diffusion.
	The generator $L$ can alternatively be written as
	\begin{equation}
		\label{eq:l:char:L1}
		Lf = B_0 f + \sum_{i=1}^d A_i(A_i f), \qquad \text{for all }f\in C^\infty(M)
	\end{equation}
	or as
	\begin{equation}
		\label{eq:l:char:L2}
		Lf = \J(dE,df) + \K(dS,df)  -  \sum_{i=1}^d A_i^*(A_i f), \qquad \text{for all }f\in C^\infty(M),
	\end{equation}
	where $A_i^*$ is again the adjoint vector field of $A_i$ in $L^2(M,\nu)$.

	Moreover, Definitions~\ref{def:mfGSDE-generator} and~\ref{def:mfGSDE-SDE} are equivalent whenever the solutions of the SDE have infinite lifetime.
\end{lemma}
\begin{proof}
	To show~\eqref{eq:l:char:L2} we use the first identity in~\eqref{eq:char-DeltaH-Ai}; to show~\eqref{eq:l:char:L1} we additionally use the characterization of $A_i^*$ after~\eqref{eq:char-DeltaH-Ai}. The equivalence of the formulations was already described in Section~\ref{s:sdes-on-mfs}.
\end{proof}

\begin{remark}
	By a reduction to coordinate charts we find that sample paths are almost surely continuous.
\end{remark}

\subsection{Properties of the process}

\label{ss:props-of-sde-mf}
In the introduction we formulated two important properties of the process~\eqref{eq:generic-sde} that we wish to preserve in the manifold generalization, the almost-sure conservation of $E$ and the stationarity of $e^{S(x)}dx$. Historically, there has been an implicit assumption that under a coordinate change from $x$ to $x = \phi(y)$ the Lebesgue measure ``$dx$'' transforms into another copy of the Lebesgue measure ``$dy$'', i.e.\ that the measure $e^{S(x)}dx$ transforms into $e^{S(\phi(y))}dy$. Since these two measures are actually different, this has led e.g.\ \"Ottinger to propose to correct for the difference by changing $S$~\cite[(1.163)]{Oettinger05}.

As described in Section~\ref{sec:generic-sde}, in this paper we take a different approach and introduce the volume form $\nu$ instead: we require that the measure $e^{S(x)}\nu(dx)$ is preserved by the flow. In this section we show that both this property and the conservation of $E$ hold for the manifold SDE~\eqref{eqdef-general:SDE-on-M}. In fact, the combination implies an even stronger property: any measure of the form $h(E)\, e^{S}\nu$ is preserved by the flow.

\medskip

\begin{lemma}[Conservation of $E$]
	Let $X$ be the solution of the SDE~\eqref{eqdef:GENERIC-sde-mf} according to Definition~\ref{def:solution-of-SDE} with maximal time of existence $\zeta$. We then have almost surely
	\[
		E(X_t) = E(X_0) \qquad \text{for all }0\leq t<\zeta.
	\]
\end{lemma}

\begin{proof}
	We first note that the vector fields $A_i$ preserve $E$:
	\[
		A_i(dE) \stackrel{\eqref{eqdef:gradient-H}}= g_\calH (\nabla_\calH E, A_i) \stackrel{(*)}= 0,
	\]
	where $(*)$ follows from the non-interaction conditions \eqref{eqdef:NIC-mf} and the definition of the metric $g_\calH$ by observing that for all $\alpha_i\in\Gamma(T^*M)$ such that $\ssharp \alpha_i = A_i$, we have
	\[
		g_\calH (\nabla_\calH E, A_i) = g_\calH (\ssharp dE, \ssharp \alpha_i) = \K(dE,\alpha_i) = 0.
	\]
	For the same reason,
	\[
		A_0(dE) = \sum_{i=1}^d \div_\nu(A_i) A_i(dE) = 0.
	\]
	The same then holds for the vector field $B_0$, since
	\[
		B_0(dE) = \underbrace{\J(dE,dE)}_{=0 \text{ since alternating}} + \;\underbrace{\K(dS,dE)}_{=0 \text{ by }\eqref{eqdef:NIC-mf}}\;  +\; A_0(dE) = 0.
	\]
	We then apply the It\^o-Stratonovich identity~\eqref{eqdef:soln-of-SDE-Ito-formula}  to the function $E$ to find for $0\leq t < \zeta$
	\[
		E(X_t) -E(X_0) = \int_0^t (dE)_{X_s} B_0(X_s) \, ds
		+
		\sqrt 2 \sum_{i=1}^d \int_0^t (dE)_{X_s} A_i(X_s) \strat dW^i_s = 0 .
		\qedhere
	\]
\end{proof}

\begin{remark}[$\calH$ is never bracket-generating]
	\label{rem:energy-is-conserved}
	Since $E$ is almost surely preserved, the process $X$ is not not transitive, but remains confined to  level sets of $E$. Consequently, in this setup the distribution $\calH$ is never bracket-generating. This is also apparent in a crucial step of the proof. Since all $A_i$s preserve $E$, also their commutators would, and thus, for a non integrable distribution, this would mean that all vectors in the tangent space preserve $E$.  In other words, when the distribution is nonintegrable, $g_\calH (\nabla_\calH E, \nabla_\calH E) = 0$ if and only if $E$ is constant, making the whole construction either trivial or impossible. However, nothing in our construction prevents the restricted distribution on a level set of $E$ to be bracket-generating.
\end{remark}

\begin{lemma}[Invariance of all measures $h(E)e^{S}\nu$]
	\label{l:invariance-ebetaS}
	Assume that the SDE~\eqref{eqdef:GENERIC-sde-mf} has infinite lifetime.
	Then for any $h\in C^\infty(E)$ it   preserves the measure $h(E)e^{S}\nu$ on~$M$.
\end{lemma}

In particular, the case $h\equiv 1$ implies that the measure $e^{S}\nu$ is preserved.

\begin{proof}
	Fix $h\in C^\infty(\R)$.
	We prove the invariance of $h(E)e^{S}\nu$ by showing that for all $f\in C_c^\infty(M)$ we have
	\begin{equation}
		\label{eq:ebetaS-invariant}
		\int_M (Lf)\, h(E)\,e^{S} \nu = 0.
	\end{equation}
	Writing
	\begin{equation}
		\label{eq:to-prove:invariance-ebetaS}
		\int_M (Lf)\, h(E)\, e^{S} \nu
		= \int_M \bra[\Big]{\J(dE,df) + \K(dS,df)}\, h(E)\, e^{S}\nu \;-\; \sum_{i=1}^d \int_M A_i^*(A_if) \,h(E)\,e^{S}\nu,
	\end{equation}
	we first note that the integral involving $\J$ vanishes by the assumptions of skew-symmetry, non-interaction~\eqref{eqdef:NIC-mf}, and unimodularity~\eqref{eq:ass:unimodularity-mf}:
	\begin{multline*}
		\int_M \J(dE,df)\,h(E)\,e^{S}\nu
		= \underbrace{\int_M \J\bra[\big]{dE,d(fh(E)e^{S})}\,\nu}_{=0\text{ by Proposition~\ref{prop:unimodularity}}}
		\;-  \int_M \underbrace{\J(dE,dS)}_{=0\text{ by }\eqref{eqdef:NIC-mf}}\,f\,h(E)\,e^{S}\nu \\
		-   \int_M \underbrace{\J(dE,dE)}_{=0\text{ by skew symmetry}}\,f\,h'(E)\,e^{S}\nu =0.
	\end{multline*}
	By rewriting the final term in~\eqref{eq:to-prove:invariance-ebetaS} as
	\begin{align*}
		- \sum_{i=1}^d \int_M A_i^*(A_if) \,h(E)\,e^{S}\nu
		 & = -  \sum_{i=1}^d \int_M (A_if) \bra[\big]{A_i (h(E)e^{S})}\,\nu             \\
		 & \leftstackrel{\eqref{eq:relation-Ai-K}}= -  \int_M \K(df,d(h(E)e^{S}))\, \nu \\
		 & = - \int_M \K(df,dS) \,h(E)\, e^{S} \nu
		-  \int_M \underbrace{\K(df,dE)}_{=0\text{ by }\eqref{eqdef:NIC-mf}} \,h'(E)\, e^{S} \nu,
	\end{align*}
	we observe that it cancels  with the integral of $\K$ in~\eqref{eq:to-prove:invariance-ebetaS}, implying~\eqref{eq:ebetaS-invariant}.
\end{proof}

When expressing the generator of the metriplectic process in \eqref{eq:generator-mf} and \eqref{eq:l:char:L2}, the metric structure seems to have a dual role, defining the diffusion part of the generator via $\Delta_\calH$, and also explicitly appearing in the drift part via $\K(dS,\cdot)$.
This is an artifact of the choice of volume~$\nu$, instead of a preserved volume from Lemma~\ref{l:invariance-ebetaS}.
We can in fact, absorb all terms related to the metric structure into a single sub-Laplacian operator, defined with respect to a the volume form $e^S\nu$. Define
\begin{equation}
	\label{eq:def:DeltaH-betaS}
	\Delta_{\calH,S} f := \Delta_\calH f +  g_\calH(\nabla_\calH f ,\nabla_\calH S)
	= \Delta_\calH f +  \K(df,dS).
\end{equation}
Then \eqref{eq:generator-mf} can be rewritten as
\[
	L f = \J(dE,df) + \Delta_{\calH,S} f.
\]
Following the earlier remark that $\Delta_\calH$ can be interpreted as the derivative of the Dirichlet form~\eqref{eqdef:Dirichlet-form-1}, we can interpret $\Delta_{\calH,S}$ as the derivative of the Dirichlet form
\[
	\calE_{S}(f, h) = \frac12 \int_M g_\calH(\nabla_\calH f, \nabla_\calH h) \, e^{S} \nu
\]
on $L^2(M; e^{S} \nu)$,
emphasizing the deep centrality of the invariant volume $e^{S}\nu$ in driving the diffusion part of the dynamics.

\subsection{The Euclidean case}

\subsubsection{Reduction to the Euclidean case}
We now show that when the manifold $M$ is $\R^d$ and the volume form $\nu$ is the Euclidean one, the geometric GENERIC SDE~\eqref{eqdef:GENERIC-sde-mf} reduces to the Euclidean SDE~\eqref{eq:generic-sde}. In fact we show a stronger statement, in which the volume $\nu$ is arbitrary, and we comment on how this formulation transforms in $\R^d$.

Consider therefore the setup of this section with $M=\R^d$, and with an arbitrary smooth non-degenerate volume form $\nu$.
We will show that the Euclidean, It\^o-interpretation version of~\eqref{eq:generic-sde} (see also~\cite[Rem.~4.6]{MielkePeletierZimmer25}) is
\begin{equation}\label{eq:generic-sde-nu}
	dX_t = \bra[\Big]{J(X_t)\, DE(X_t) + K(X_t)\, DS(X_t) +  (\div_\nu K)(X_t)}\, dt + \Sigma(X_t) \, dB_t.
\end{equation}
Note the only difference: the `flat' divergence $\div K$ in~\eqref{eq:generic-sde} has been replaced by the $\nu$-divergence $\div_\nu K$ above. When $\nu$ is the standard volume form on $\R^d$ these two coincide.

We now show that the SDE~\eqref{eqdef:sde-mf-with-JK} reduces to~\eqref{eq:generic-sde-nu}.
The vector fields $J(dE,\cdot)$ and $K(dS,\cdot)$ on the right-hand side of~\eqref{eqdef:sde-mf-with-JK} clearly reduce to their counterparts in~\eqref{eq:generic-sde-nu}, and we therefore only need to show that the terms
\begin{equation}
	\label{eq:dissipative-part-sde-mf}
	\sum_{i=1}^d (\div_\nu  A_i) A_i\, dt + \sqrt{2}\sum_{i=1}^d A_i(X_t) \strat dW^i_t
\end{equation}
reduce to the corresponding terms in~\eqref{eq:generic-sde-nu},
\begin{equation}
	\label{eq:dissipative-part-sde-Rd}
	\div_\nu K(X_t)\, dt + \Sigma(X_t) \, dW_t.
\end{equation}
Note that the Stratonovic interpretation of the first expression has changed into the It\^o interpretation of the second.

We start by interpreting the Euclidean matrix field $\Sigma$ in~\eqref{eq:dissipative-part-sde-Rd} in terms of the $A_i$ that appear as mobilities in~\eqref{eqdef:two-versions-of-gGENERIC-SDE}, by setting for any $f\in C^\infty(\R^d)$
\begin{equation}
	\label{eqdef:Sigma-in-terms-of-A_i}
	\Sigma(x) D f(x) :=  \sqrt{2} \sum_{i=1}^d e_i\, (A_if)(x),
\end{equation}
where the $e_i$ are an orthonormal basis of $\R^d$ in the standard metric. It follows that for any $f$ and $h$
\[
	D f^* \Sigma\Sigma^* D h
	\stackrel{\eqref{eqdef:Sigma-in-terms-of-A_i}}= 2\sum_{i=1}^d (A_if)(A_i h) \stackrel{\eqref{eq:char:DeltaH-Ais}}= 2K(df,dh) = 2D f^* K D h,
\]
which implies that this choice of $\Sigma$ satisfies the fluctuation-dissipation relation~\eqref{eq:fluctuation-dissipation}.

We now show that~\eqref{eq:dissipative-part-sde-mf} and~\eqref{eq:dissipative-part-sde-Rd} coincide. We write $a_{ij}(x)$ for the coordinates in $\R^d$ of the vector fields $A_i(x)$. For each $i$, by the standard translation between Stratonovich and It\^o interpretations (e.g.~\cite[Sec.~6.5.6]{Evans_2013}) the Stratonovich term
\[
	\sqrt{2} A_i \strat dW^i = \sqrt{2} \sum_{j=1}^d e_j \;a_{ij}\strat dW^i
\]
has an equivalent formation as It\^o term
\[
	\sqrt{2} A_i\, dW^i +  c \, dt, \qquad \text{where}\qquad
	c_j = \sum_{i=1}^d \sum_{k=1}^d \bra[\big]{\partial_{x_k} a_{ij}} a_{ik}.
\]
Calculating the $j^{\mathrm{th}}$ coordinate of $\div_\nu K$, we observe that
\begin{align*}
	(\div_\nu K)_j
	 & =   \frac1\nu \sum_{i=1}^d \sum_{k=1}^d \partial_{x_k} (\nu\, a_{ij}a_{ik})
	=
	\sum_{i=1}^d \sum_{k=1}^d \bra[\big]{\partial_{x_k} a_{ij}}a_{ik}
	+  \frac1\nu \sum_{i=1}^d \sum_{k=1}^d a_{ij}\bra[\big]{\partial_{x_k} (\nu \, a_{ik})} \\
	 & =  c_j +   \sum_{i=1}^d a_{ij} \div_\nu  A_i,
\end{align*}
and therefore
\begin{align*}
	\sum_{i=1}^d (\div_\nu  A_i) A_i\, dt + \sqrt{2}\sum_{i=1}^d A_i(X_t) \strat dW^i_t
	 & =  \div_\nu K -  c + \sqrt{2} \sum_{i=1}^d A_i dW^i +   c \\
	 & =   \div_\nu K  + \sqrt{2} \sum_{i=1}^dA_i dW^i.
\end{align*}
This coincides with~\eqref{eq:dissipative-part-sde-Rd} by our definition of $\Sigma$, showing that the two formulations coincide when the manifold $M$ is $\R^d$ and $\nu$ is the Lebesgue measure.

\subsubsection{Coordinate transformations in the Euclidean GENERIC SDE}
\label{sss:coordinate-transformations-euclidean}
The general-$\nu$ SDE~\eqref{eq:generic-sde-nu} on $\R^d$ is coordinate invariant.
In~\cite[Rem.~4.6]{MielkePeletierZimmer25} the rules for transformation of coordinates for this SDE were already given; we repeat them here for convenience.
If $X_t$ is a solution of~\eqref{eq:generic-sde-nu}   and $\phi\colon\R^d\to\R^d$ is a smooth bijection, then $Y_t :=\phi(X_t)$ is a solution of
\begin{equation}
	\label{eq:GSDE-coordinate-invariant-hat}
	d  Y_t = \bra*{\wh J(Y_t) D \wh E(Y_t) + \wh K(Y_t) D \wh S(Y_t)
		+   (\div_{\wh \nu}{\wh K})(Y_t)}\, d t
	+ \wh\Sigma(Y_t) \, d W_t.
\end{equation}
Here the components are defined as
\begin{alignat}{2}
	\wh E(\phi(x))   & :=   E(x),                         & \qquad \wh S(\phi(x))                     & :=  S(x),\notag                                             \\
	\wh J(\phi(x))   & :=  D \phi(x) J(x)  D  \phi(x)^*,  & \wh  K(\phi(x))                           & :=  D \phi(x)  K(x)  D  \phi(x)^*,  \label{eqdef:Jhat-Khat} \\
	\wh \nu(\phi(x)) & := \frac 1{\det D \phi(x)} \nu(x), & \quad\text{and}\quad \hat \Sigma(\phi(x)) & :=  D \phi(x) \Sigma(x).\notag
\end{alignat}
Expressions such as in~\eqref{eqdef:Jhat-Khat} should be read as
\[
	\wh J_{ij}(\phi(x)) := \sum_{k,\ell} \partial_{x_k}\phi_i(x) J_{k\ell}(x) \partial_{x_\ell} \phi_j(x).
\]
One can check that
\begin{enumerate}
	\item $\wh J$ and $\wh \nu$  again satisfy the unimodularity condition~\eqref{eq:ass:unimodularity-mf}, i.e.\ $\div_{\wh \nu} (\wh J(dh, \cdot)) = 0$;
	\item $\wh \nu$ is the push-forward of the volume form $\nu$  under $\phi$;
	\item The volume $\wh \nu(y) \ee^{\wh S(y)}$ is the push-forward of $\nu( x)  \ee^{  S(x)}$ under~$\phi$, and is stationary for~\eqref{eq:GSDE-coordinate-invariant-hat}.
\end{enumerate}

\subsection{The Fokker-Planck equation}

The generator~\eqref{eq:generator-mf} also gives rise to the Fokker-Planck equation for the law of the process $X$. If we write $\rho_t\nu$ for this law, where $\rho_t$ is a smooth (time-varying) scalar function on $M$, then the evolution of $\rho$ is given by the equation
\[
	\dot\rho = L^* \rho,
\]
where $L^*$ is an adjoint of $L$ in an appropriately chosen space. This adjoint takes a particularly simple form if we choose the $L^2$-space with weight $e^{S}\nu$ and recall \eqref{eq:def:DeltaH-betaS}:
\begin{lemma}[Adjoint of $L$]
	Writing
	\[
		Lf = \J(dE,df) + \K(dS,df) + \Delta_\calH f = \underbrace{\J(dE,df)}_{=:\, L_af} + \underbrace{\Delta_{\calH, S} f}_{=:\,L_s f}
	\]
	the adjoint $L^*$ in $L^2(M;e^{S}\nu)$ of $L$ has the characterization
	\begin{equation}
		\label{eq:l:char-L*}
		L^*f = -L_af + L_s f = -\J(dE,df) + \Delta_{\calH, S} f, \qquad \text{for }f\in C_c^\infty(M),
	\end{equation}
	i.e.\ in the space $L^2(M;e^{S}\nu)$ the separate operators $L_s$ and $L_a$ are symmetric and antisymmetric, respectively.
\end{lemma}

\begin{proof}
	The antisymmetry of $L_a$ follows from remarking that for any functions $f$ and $h$,
	\begin{align*}
		\int_M h \J(dE,df)\, e^{S}\nu & = \underbrace{\int_M \J(d(Ee^{S}),d(fh))\,\nu}_{=0 \text{ by Proposition~\ref{prop:unimodularity}}}
		-\int_M E \,\J(d(e^{S}), d(fh))\, \nu
		- \int_M f \J(dE,dh)\, e^{S}\nu                                                                                                     \\
		                              & = -\int_M E \underbrace{\J(dS, d(fh))}_{=0 \text{ by \eqref{eqdef:NIC-mf}}}\,e^{S}\nu
		- \int_M f \J(dE,dh)\, e^{S}\nu.
	\end{align*}
	We have already shown the symmetry of $L_s$ in \eqref{eq:char:DeltaH-Ais}, where we only need to replace $\nu$ with~$e^{S}\nu$.
\end{proof}

This formulation gives an newfound relevance to the volume form $e^{S}\nu$, which does no longer appear as an informed but somewhat arbitrary modeling choice, but seems to underline an essential structure of the system.
Indeed, the laplacian that drives the diffusion in the Fokker-Plank equation, which is usually described by separate metric and diffusion components, is now a single geometric laplacian, symmetric with respect to this new deformed volume. This also highlight how the entropy is not just an arbitrary driving functional, but is in some sense deeply intertwined with the geometry of the diffusion.

\subsection{The Fokker-Planck equation is a deterministic GENERIC equation in its own right}

The Fokker-Planck equation $\dot \rho = L^*\rho$ for the law of a gGENERIC SDE is itself a GENERIC equation. This was observed in~\cite{GrmelaOttinger97} for the related \emph{GENERIC with fluctations}, and here we describe this structure in a heuristic manner for the geometric GENERIC SDE of this paper.

As in the previous section, we write the law of a gGENERIC SDE as $\rho_t\nu$, where $\rho_t$ is a time-varying element of $L^2(M;\nu)$. At the level of laws, the state space is therefore the infinite dimensional manifold $\bbM := L^2(M;\nu)$, whose elements we will denote $\rho$. Since we now are dealing with two manifolds, $M$ and $\bbM$, we continue to use `$d$' for differentials on $M$ but we write `$d_\bbM$' for differentials on $\bbM$.

Given the setup of the gGENERIC SDE at the beginning of this section, in terms of $(M,\K,\J,\nu,E,S)$,
we now define the four components $\bbE$, $\bbS$, $\bbJ$, and $\bbK$ of a deterministic GENERIC system:
\begin{subequations}
	\label{eqdef:gGENERIC-setup-FP-eq}
	\begin{align}
		\bbE(\rho)           & := \int_M E \, \rho\, \nu,               & \bbS(\rho)           & = \int_M (S - \log \rho)\,\rho \,\nu ,   \\
		\bbJ_\rho(\bbA,\bbB) & = \int_M \J(d \bbA, d\bbB)\, \rho\, \nu, & \bbK_\rho(\bbA,\bbB) & = \int_M \K(d \bbA, d\bbB)\, \rho\, \nu.
	\end{align}
\end{subequations}
Here $\bbA$ and $\bbB$ are covector fields on $\bbM$, which we identify with scalar functions on $M$; this can be done for the case that we are interested in, which is the case when $\bbA$ and $\bbB$ are differentials $d_\bbM \bbF$ and $d_\bbM \bbG$ of some functions $\bbF$ and $\bbG$ on $\bbM$. In this case we consider $\bbA = d_\bbM \bbF(\rho)$ to have an alternative (and more explicit) interpretation as
\[
	\int_M \bbA(x) \tilde \rho(x) \nu(dx) := \lim_{\e\to 0} \frac1\e\bra[\big]{\bbF(\rho+\e\tilde\rho)-\bbF(\rho)}
	\qquad\text{for all }\tilde \rho.
\]
For instance, in the derivation below we will choose $\bbG(\rho) := \int_M \varphi \, \rho\, \nu$ for some $\varphi\in C^\infty(M)$, with which $d_{\bbM} \bbG(\rho) = \varphi$ for any $\rho$. Note that by this identification we have
\[
	d_\bbM \bbE (\rho) =E
	\qquad \text{and}\qquad
	d_\bbM \bbS(\rho) = S - - \log\rho - 1.
\]

We now show that the deterministic GENERIC system determined by $(\bbM, \bbK,\bbJ,\bbE,\bbS)$ coincides with $\dot \rho = L^*\rho$, where we consider $L^*$ to be the adjoint in $L^2(M;\nu)$, by our choice of $\bbM$. Note that this is different from the adjoint in $L^2(M;e^{S}\nu)$ that we considered in the previous section.
Note that for any $\rho,\varphi\in C^\infty(M)$ we have
\begin{align*}
	\int_M \rho (\Delta_{\calH} \varphi )\,\nu
	 & \stackrel{\eqref{eq:sub-Laplacian}}=  \int_M \rho \div_\nu (\nabla_\calH \varphi)\, \nu
	\stackrel{\eqref{eq:int-by-parts-K}}= -\int_M \K(d\rho,d\varphi)\, \nu
	= -\int_M \K(d(\log\rho),d\varphi)\, \rho\, \nu,
\end{align*}
and therefore
\begin{align*}
	\frac d{dt} \bbG(\rho_t)
	 & = \int_M \varphi L^* \rho_t \, \nu
	= \int_M \rho_t L\varphi \, \nu                                                                                               \\
	 & = \int_M \rho_t \J(dE,d\varphi)\, \nu + \int_M \rho_t \K(dS,d\varphi)\, \nu + \int_M \rho_t (\Delta_{\calH} \varphi )\,\nu \\
	 & = \bbJ_{\rho_t} (d_\bbM \bbE, d_\bbM \bbG)
	+ \int_M \K(d(S-\log \rho_t ),d \varphi)\, \rho_t\, \nu                                                                       \\
	 & = \bbJ_{\rho_t} (d_\bbM \bbE, d_\bbM \bbG)
	+ \bbK_{\rho_t}(d_\bbM \bbS, d_\bbM \bbG).
\end{align*}
It follows that the curve $t\mapsto \rho_t$ in $\bbM$ is a solution of the deterministic GENERIC equation in the form~\eqref{eq:gGENERIC}, with components $(\bbM, \bbK,\bbJ,\bbE,\bbS)$.

\begin{remark}[Role of $\nu$]
	In the setup~\eqref{eqdef:gGENERIC-setup-FP-eq} the choice of the volume $\nu$ appears to be uninportant, since $\nu$  appears every time in conjuction with $\rho$. In fact, $\nu$ has an important role in the expression
	\[
		\bbS(\rho) =\int_M (S - \log \rho)\,\rho \,\nu,
	\]
	which can be interpreted as the (negative) relative entropy of $\rho\nu$ with respect to the measure~$e^{S}\nu$, or rewritten in more conventional form,
	\[
		\bbS(\rho) = -\calH(\rho\nu| e^{S}\nu)
		= -\int_M \bra[\Big]{\log \frac {d\rho\nu}{de^{S}\nu}}\,
		\rho\,\nu.
	\]
	Here one observes again the role of $\nu$ as reference measure.
\end{remark}

\section{Limit of vanishing noise}
\label{s:vanishing-noise}

One design requirement of the geometric GENERIC SDE of this paper is that it reduces consistently to the deterministic GENERIC evolution~\eqref{eq:gGENERIC} in the limit of `small noise', or equivalently of `low temperature'. Since we have not yet discussed `temperature' at all, we need to understand what this `low-temperature limit' is.

This limit refers to the situation where the system is effectively operating at constant temperature, for instance because it is in contact with a `heat bath' that enforces a spatially constant and fixed temperature $T$. By tuning this heat-bath temperature $T$ one can influence the temperature of the system, and one can also make that temperature small.

If we assume such a heat-bath-enforced, isothermal situation, then the temperature $T$ is a fixed parameter in the evolution, and the question arises how this temperature $T$ appears in the stochastic gGENERIC equations. Here we focus on a common situation in which the entropy~$S$ depends linearly on the parameter $1/T$:
\begin{equation}
	\label{ass:S-scales-like-1/T}
	S(x) = \frac1T \wt S(x) ,
\end{equation}
where $\wt S$ is independent of $T$. This can be proved  for instance in the case that the heat bath is a large collection of independent Hamiltonian oscillators; see e.g.~\cite[p.~69]{Chorin94} or~\cite[Sec.~4.6, in particular equation (4.27)]{MielkePeletierZimmer25}. As an example, consider the Langevin equation~\cite{Nelson67,Pavliotis14}, which models the movement of a tagged particle in a literal `bath' of other particles. The surrounding particles are not modelled explicitly, but their influence on the tagged particle is represented by drag  and thermal agitation forces; such a `bath' can be considered a `heat bath' because there are many more surrounding particles than tagged particles, and in approximation the bath therefore preserves its energy per particle and its temperature.

\medskip

In this situation it is common to assume that the Onsager operator $\K$ depends linearly on~$T$:
\begin{equation}
	\label{ass:K-scales-like-T}
	\K(\cdot,\cdot) = T\, \wt \K(\cdot,\cdot),
\end{equation}
where again $\wt \K$ is independent of $T$. This choice connects to the fluctuation-dissipation relation~\eqref{eq:fluctuation-dissipation} that we imposed above. Derivations of this relation in the particle-in-a-bath context, going back to Einstein~\cite{Einstein05} (see e.g.~\cite[App.~A.1.10]{Sekimoto10} for a modern derivation) show this relation in the modified form
\[
	\Sigma \Sigma^*(x) = 2  T \wt K(x),
\]
where $\wt K$ is the drag coefficient of a particle in the fluid, i.e.\ the ratio of force to velocity on a particle being dragged through the environment, and it is independent of $T$.

\medskip
Putting assumptions~\eqref{ass:S-scales-like-1/T} and~\eqref{ass:K-scales-like-T} together, in view of~\eqref{eq:sub-Laplacian} and~\eqref{eq:int-by-parts-K} we find that the generator $L$ can be written as
\begin{equation}
	\label{eq:gGENERIC-L-with-T}
	Lf  = \J(dE,df) + \wt\K(d\wt S,df) + T\, \wt \Delta_\calH f \qquad\text{for any }f\in C_c^\infty(M),
\end{equation}
or equivalently the SDE~\eqref{eqdef:sde-mf-with-JK} can be written as
\begin{equation}
	\label{eq:gGENERIC-SDE-with-T}
	d X_t= \bra[\Big]{\J(dE,\cdot) + \wt\K(d\wt S,\cdot) +  T \sum_{i=1}^d (\div_\nu  \wt A_i) \wt A_i}\, dt + \sqrt{2T}\sum_{i=1}^d \wt A_i(X_t) \strat dW^i_t,
\end{equation}
where $\wt\Delta_{\calH}$ and $\wt A_i$ should be understood as the result of applying Section~\ref{ss:sub-Riemannian-BM} to $\wt\K$ instead of~$\K$.
In particular, observe that in the limit $T\to0$ the two expressions~\eqref{eq:gGENERIC-L-with-T} and~\eqref{eq:gGENERIC-SDE-with-T} revert to their equivalents of the deterministic GENERIC equation, but now in terms of $\wt \K$ and $\wt S$.

\begin{remark}
	Note however that by the same assumptions~\eqref{ass:S-scales-like-1/T} and~\eqref{ass:K-scales-like-T}, in the \emph{deterministic} GENERIC equation we can replace the combination $\wt\K(d\wt S,\cdot)$ by $\K(dS,\cdot)$ without changing the evolution; it is only in the stochastic version that one notices the scaling in $T$ of $S$ and $\K$.

	\medskip
	Also the role of $\nu$ diminishes when $T\to0$: this can be recognized in the fact that $\nu$ only appears in the last drift term in~\eqref{eq:gGENERIC-SDE-with-T} (the third term in parentheses), and is prefixed by $T$. This also implies that in the Euclidean case the difference between the coordinate-invariant SDE~\eqref{eqdef:sde-mf-with-JK} and the non-coordinate-invariant SDE~\eqref{eq:generic-sde} vanishes as $T\to0$.
\end{remark}

\section{Conclusion and discussion}\label{sec:discussion}

In this paper we propose a geometric formulation for GENERIC stochastic differential equations on manifolds. This SDE has been constructed to satisfy a number of conditions:
\begin{itemize}
	\item It has a clear geometric structure, in which the geometric components of the manifold ($\K$, $\J$, and $\nu$) are separated as much as possible from the functions $E$ and $S$.
	      In particular, it is fully coordinate-invariant.
	\item The measure $e^{S}\nu$ is invariant under the flow.
	\item It reduces to the usual GENERIC equation in the limit of zero noise.
\end{itemize}
In addition, the particular structure has been inspired by the coarse-graining derivation of~\cite{MielkePeletierZimmer25}.
In the rest of this section we discuss a number of further aspects.

\subsection{The volume $\nu$ and the unimodularity condition}
\label{ss:unimodularity}

In this paper we introduce an explicit volume form $\nu$ on $M$ that does not appear in earlier work. We claim that this form is (a) necessary to make the setup geometrically consistent, (b) implicitly present in earlier work, and (c) physically meaningful. We now comment on these aspects.

As remarked in Section~\ref{sec:generic-sde}, a design requirement is that the measure ``$e^{S(x)}dx$'' is invariant under the SDE flow. In fact, this expression is a statement about \emph{two} measures, not one: the Lebesgue measure ``$dx$'' appears as a reference measure, and ``$e^{S(x)}dx$'' is a modification or a `tilting'~\cite{PeletierSchlichting23} of this measure.

This raises the question, `what is special about $dx$?' Here
we claim that the answer lies in the fact that canonical Hamiltonian systems in $\R^{2n}$ preserve the Lebesgue measure. To illustrate this, consider canonical Hamiltonian variables $(q,p)$ in $\R^{2n}$; then the symplectic operator $J$ has the form

\[
	J = \begin{pmatrix}
		0 & I \\ -I & 0
	\end{pmatrix}.
\]
Any Hamiltonian evolution with this symplectic operator conserves the Lebesgue measure.
After transformation to non-canonical coordinates, the Lebesgue measure is transformed into a new measure (let us call it $\nu$) and the symplectic structure $J$ is transformed to, say, $\wh J$. The conservation of the measure $\nu$ under any Hamiltonian vector field for the transformed symplectic structure should then hold for any Hamiltonian vector field $\wh J(d\wh h, \cdot)$. Since a measure $\rho$ is stationary under a vector field $X$ iff $\div _\rho X = 0$, this leads us to the unimodularity condition, $\div_\nu J(dh,\cdot)=0$ for any $h$.

In conclusion, we posit that any GENERIC SDE involves not one but \emph{two} distinguished measures:
\begin{itemize}
	\item A measure $\nu$ corresponding to an invariant volume for the unimodular Poisson structure~$J$;
	\item A measure $e^{S(x)}\nu(dx)$ which is stationary for the full SDE evolution.
\end{itemize}
When the Poisson structure is inherited from canonical variables, the measure $\nu$ typically is the Lebesgue measure, and this explains why the necessity of making $\nu$ explicit has not been recognized before.

\begin{remark}[Unimodularity is automatic in symplectic structures, not in Poisson structures]\label{rmk:univalence-geometric}
	The conservation of the Lebesgue measure above by the canonical symplectic form in $\R^{2n}$ has a direct extension to the more general case of classical (i.e.\ regular) Hamiltonian systems on the phase-space manifold $T^*M$, where it is the volume form induced by the symplectic form that is preserved by the flow.

	However, for arbitrary Poisson structure and arbitrary choices of a volume form this is not necessarily the case, requiring the introduction of unimodularity: a Poisson manifold is unimodular if there exists a volume form invariant under all its Hamiltonian flows, that is, if there exists $\nu$ such that $\calL_{X_h}\nu = 0$ for all $h\in C^\infty (M)$. For more information we refer the reader to a book on Poisson geometry, e.g.\ the recent monography~\cite{Crainic2021}.

	A careful look at the proof of Proposition~\ref{prop:unimodularity} shows that the geometry imposes constraints on what those volumes can be. Indeed, such an invariant volume should be unique up to multiplication by a nowhere-vanishing Casimir function: let $\mu$ and $\mu'$ be two invariant volume forms, and let~$h$ be a nowhere-vanishing function such that $\mu' = h\mu$. Then, following similar ideas as in the aforementioned proof we get
	\[
		0 = \mathcal{L}_{X_f}\mu' = \mathcal{L}_{X_f}(h\mu)
		= (X_f h)\,\mu + h\,\mathcal{L}_{X_f}\mu
		= (X_f h)\,\mu,
	\]
	where $\mathcal{L}_{X_f}\mu=0$ by invariance of the volume (i.e., unimodularity). Hence $X_f h=0$ for every $f$, that is, $h$ Poisson-commutes with every function and thus $g$ is a Casimir.
	This has interesting implications. For instance, on a connected nondegenerate Poisson manifold the Casimirs are just constants~\cite[Example 9.31]{Crainic2021}, and in such case the invariant volume is unique up to a constant factor.
	The study of these problems and characterizations is part of Poisson cohomology, and more information can be found in the relevant literature, e.g.~\cite[Chapter 9ff]{Crainic2021}.
\end{remark}

\subsection{Other approaches to geometrization of GENERIC}
\label{ss:Esen-et-al}

Esen, Grmela, and Pavelka~\cite{EsenGrmelaPavelka22I,EsenGrmelaPavelka22II} propose a symplectic and a contact-geometric formulation of the GENERIC equation, by embedding the state space of the GENERIC equation~\eqref{eq:generic} into a space of twice the dimension (`doubling the variables').

\medskip

Our approach is different in three main aspects:
\begin{itemize}
	\item we aim at a geometrization that includes the stochastic evolution from the beginning;
	\item we work directly on the state space of the system, without introducing additional state-space dimensions;
	\item we focus on a description that keeps the Poisson and metric structures on the same footing, without privileging one over the other.
\end{itemize}
The latter two points are crucial. Given any vector field, one can always find a symplectic or contact structure that reproduces it by extending the state space with additional dimensions. However, this construction is not unique, and it does not clarify the role of the different structures at play. In our opinion, a geometric formulation should provide such clarification, and this is part of what we aim to achieve in this paper, under the guiding principles explained at the beginning of this section.

\subsection{Relation with coarse-graining}

In the early introductions of GENERIC in e.g.~\cite{Ottinger98} and~\cite[Ch.~6]{Oettinger05}, \"Ottinger provides several motivations for the GENERIC structure and its restrictions that are based on coarse-graining. From this point of view, the GENERIC equation arises as the coarse-grained, upscaled version of a more microscopic evolution equation, which could be purely Hamiltonian or yet another GENERIC equation. In this `coarse-graining process' information is lost: the microscopic state space is `larger' in some sense than the upscaled state space, and this loss of information, or loss of resolution, generates additional components in the entropy $S$.

In~\cite{MielkePeletierZimmer25} the authors follow the same path, with the aim of motivating the GENERIC structure through coarse-graining. Their coarse-graining procedure is slightly different, and their aim is to perform the coarse-graining in a mathematically rigorous way. Our choice in this paper to take the unimodularity condition as foundational is inspired by the findings of~\cite{MielkePeletierZimmer25}.

However, our postulate of~\eqref{eq:ass:unimodularity-mf} still is only a postulate, and we still want to motivate this, rigorously, by a coarse-graining procedure. In fact, despite the efforts in~\cite{Oettinger05,MielkePeletierZimmer25}, the understanding of GENERIC through coarse-graining still is far from complete. This is a topic that we will be pursuing in the future.

\subsection{Freedom of choice in mathematical and modelling senses}
\label{ss:math-vs-modelling}

The systems that we study in this paper are models of real-world systems, and consequently all the components of a gGENERIC SDE---the whole set $(M,\K,\J,\nu,E,S)$---can and should be interpretable in a modelling sense. For some of these components, such as the state space $M$ and the energy $E$, this interpretation is straightforward, while for others, such as the entropy $S$ and the Onsager operator $\K$, this may be more opaque. In the case of the volume form $\nu$ we have discussed this aspect in Section~\ref{ss:unimodularity}.

A corollary of this point of view is that the \emph{mathematical} structures that we construct in this paper should use exactly these modelling components, and nothing else. This is indeed the case, and since this fact is slightly hidden, it is worth making it explicit.

In our construction we define the Laplacian operator and the diffusion process by choosing a volume (see Section~\ref{ss:sub-Riemannian-BM}). However, there is an alternative way to define the Laplacian on a sub-Riemannian manifold, one which does not require choosing a volume.
Instead one chooses a partial connection on the manifold, and then defines the Laplacian as the trace of the Hessian (see e.g.~\cite[Section 4.2]{Thalmaier16}). This can be convenient as it allows one to define horizontal martingales by lifting them from Euclidean martingales via their stochastic developments, generalizing the similar procedure from the Riemannian case (see also the explanation in~\cite{GrongThalmaier16}).
So one could question why our approach in this paper should be the preferred one.

To answer this, we need to delve a bit more into the details of this other construction. A choice of connection, in this case, is effectively a choice of the vertical bundle $\calV$, such that $TM = \calH \oplus \calV$. This can be done for instance by choosing a Riemannian metric $\mathbf{g}$ on $M$ that restricts as $\mathbf{g}|_\calH = g_\calH$, and then taking $\calV$ as the orthogonal complement of $\calH$. If $\nabla$ is the Levi-Civita connection of $\mathbf{g}$, one can check that for any two horizontal vectors $X,Y \in \calH$, $\pi_\calH \nabla_X Y$ does not depend on $\mathbf{g}|_\calV$, and thus one can appropriately nest projections in the definition of the Hessian to obtain a well-defined Laplacian operator.
While this leads to the Laplace-Beltrami operator in the Riemannian case, in the sub-Riemannian case this can lead to different operators depending on the choice of $\mathbf{g}$, and moreover there are different choices of $\mathbf{g}$ that lead to the same operator.
This would introduce a non-trivial choice in the construction, which would not be interpretable in a modelling sense and which leaves a freedom (due to the aforementioned non-uniqueness) that would be hard to justify in practice.

So, while this alternative approach can be useful in spectral geometry approaches to sub-Riemannian diffusion, in our context the choice of a volume form is more natural and coincides precisely with the modelling point of view of the equations.

\section*{Acknowledgements}
We are grateful to Miroslav Grmela and Alessandro Bravetti for their useful observations and comments on the first draft of this manuscript.

\bibliographystyle{alphainitials}
\bibliography{bib-merged}

\end{document}